\documentclass[12pt,reqno]{amsart}

\usepackage[latin1]{inputenc}
\usepackage{amsmath}
\usepackage{amsfonts}
\usepackage{amssymb}
\usepackage{graphics}
\usepackage{enumerate}
\usepackage{amssymb,amsmath,amsthm,amscd,latexsym,verbatim,graphicx,amsfonts}

\usepackage{amscd}
\usepackage{amsmath}
\usepackage{amssymb}
\usepackage{amsthm}
\usepackage{latexsym}
\usepackage{verbatim}

\theoremstyle{plain}
\newtheorem{theorem}{Theorem}[section]

\newtheorem{lemma}[theorem]{Lemma}

\theoremstyle{definition}

\theoremstyle{remark}

\newcommand{\nri}{n\rightarrow\infty}
\newcommand{\kri}{k\rightarrow\infty}
\newcommand{\tri}{t\rightarrow\infty}

\newcommand{\bbC}{\mathbb{C}}

\newcommand{\bbD}{\mathbb{D}}

\newcommand{\bbN}{\mathbb{N}}

\newcommand{\mca}{\mathcal{A}}
\newcommand{\mcj}{\mathcal{J}}

\newcommand{\mcl}{\mathcal{L}}
\newcommand{\mcm}{\mathcal{M}}
\newcommand{\mcr}{\mathcal{R}}

\newcommand{\eitheta}{e^{i\theta}}

\DeclareMathOperator*{\supp}{supp}

\title[]{Hyponormal Toeplitz Operators on Weighted Bergman Spaces}
\author[]{Brian Simanek}
\date{}

\begin{document}
\maketitle

\begin{abstract}
We consider the Toeplitz operator with symbol ${z^n+C|z|^s}$ acting on certain weighted Bergman spaces and determine for what values of the constant $C$ this operator is hyponormal.  The condition is presented in terms of the norm of an explicit block Jacobi matrix.
\end{abstract}

\vspace{4mm}


\vspace{2mm}


\vspace{2mm}

\normalsize

\section{Introduction}\label{Intro}

\subsection{Weighted Bergman Spaces}\label{wberg}
Let $\mu$ be a probability measure on the interval $[0,1]$ with $1\in\supp(\mu)$ and $\mu(\{1\})=0$.  Using $\mu$, define the measure $\nu$ on the open unit disk $\bbD$ by $d\nu(r\eitheta)=d\mu(r)\times\frac{d\theta}{2\pi}$.  Let $\mca^2_{\nu}(\bbD)$ denote the weighted Bergman space of the unit disk defined by
\[
\mca^2_{\nu}(\bbD)=\left\{f:\int_{\bbD}|f(z)|^2d\nu(z)<\infty,\, f\mathrm{\, is\, analytic\, in\, }\bbD\right\}
\]
We equip $\mca^2_{\nu}(\bbD)$ with the inner product
\[
\left\langle f,g\right\rangle_{\nu}=\int_{\bbD}f(z)\overline{g(z)}\,d\nu(z).
\]
Notice that the rotation invariance of the measure $\nu$ means the monomials $\{z^n\}_{n=0}^{\infty}$ are an orthogonal set in $L^2(\bbD,d\nu)$ and $\mca_{\nu}^2(\bbD)$.  It is a standard fact that $\mca_{\nu}^2(\bbD)$ is a reproducing kernel Hilbert space.  Let us define the set $\{\gamma_t\}_{t\in[0,\infty)}$ by
\[
\gamma_t:=\int_{\bbD}|z|^td\nu(z)=\int_{[0,1]}x^td\mu(x).
\]
Since $1\in\supp(\mu)$, the sequence $\{\gamma_n\}_{n\in\bbN}$ decays subexponentially as $\nri$, meaning for all $t>0$ it holds that $\gamma_{m+t}/\gamma_m\rightarrow1$ as $m\rightarrow\infty$.  Since $\mu(\{1\})=0$, we know $\gamma_t$ approaches $0$ as $\tri$.  With this notation it is true that
\[
\mca^2_{\nu}(\bbD)=\left\{f(z)=\displaystyle\sum_{n=0}^{\infty}a_nz^n:\displaystyle\sum_{n=0}^{\infty}|a_n|^2\gamma_{2n}<\infty\right\}
\]
and the inner product becomes
\[
\left\langle\displaystyle\sum_{n=0}^{\infty}a_nz^n,\displaystyle\sum_{n=0}^{\infty}b_nz^n\right\rangle_{\nu}=\displaystyle\sum_{n=0}^{\infty}a_n\bar{b}_n\gamma_{2n}.
\]
Of particular interest is the case when
\[
d\nu(z)=(\beta+1)(1-|z|^2)^{\beta}dA
\]
for some $\beta\in(-1,\infty)$, where $dA$ is normalized area measure on $\bbD$ (see \cite{BKLSS,HwangLee,HLP,LuLiu,LuShi}).  Notice that when $\beta=0$, the space $\mca_\nu(\bbD)$ is just the usual Bergman space of the unit disk.

A bounded operator $T$ acting on a Hilbert space is said to be \textit{hyponormal} if $[T^*,T]\geq0$, where $T^*$ denotes the adjoint of $T$.  The motivation for studying such operators comes from Putnam's inequality (see \cite[Theorem 1]{Putnam}), which says that hyponormal operators satisfy
\[
\|[T^*,T]\|\leq\frac{|\sigma(T)|_2}{\pi}
\]
where $\sigma(T)$ is the spectrum of $T$ and $|\cdot|_2$ dentoes the two-dimensional area.

If $\varphi\in L^{\infty}(\bbD)$, then we define the operator $T_{\varphi}:\mca^2_{\nu}(\bbD)\rightarrow \mca^2_{\nu}(\bbD)$ with symbol $\varphi$ by
\[
T_{\varphi}(f)=P_{\nu}(\varphi f),
\]
where $P_{\nu}$ denotes the orthogonal projection to $\mca^2_{\nu}(\bbD)$ in $L^2(\bbD,d\nu)$.  There is an extensive literature aimed at characterizing those symbols $\varphi$ for which the corresponding operator $T_{\varphi}$ is hyponormal, much of which focuses on the special case of the classical Bergman space of the unit disk (see \cite{AC,CC,FL,Hwang,Hwang2,HwangLee,HLP,LuLiu,LuShi,Sadraoui,SimHypo}).  The specific symbol we will focus on is $\varphi(z)=z^n+C|z|^s$, where $n\in\bbN$, $C\in\bbC$, and $s\in(0,\infty)$.  The case $n=1$ and $s=2$ in the classical Bergman space was considered in \cite{FL} while a broader range of $n$ and $s$ was previously considered in \cite{SimHypo}, where it was shown that hyponormality of $T_{\varphi}$ acting on the classical Bergman space implies $|C|\leq\frac{n}{s}$ and the converse holds if $s\geq2n$.  Theorem \ref{better1} below will complete that result by providing necessary and sufficient conditions on the constant $C$ for $T_{\varphi}$ acting on any $\mca_{\nu}^2(\bbD)$ to be hyponormal.  As a result, we will recover the aforementioned result from \cite{SimHypo}.  Our condition is stated in terms of the norm of a certain self-adjoint operator that happens to be a block Jacobi matrix.



\subsection{Block Jacobi Matrices}\label{block}

Block Jacobi matrices are matrices of the form
\[
\mcm=\begin{pmatrix}
B_1 & A_1 & 0 & \cdots & \cdots\\
A_1^* & B_2 & A_2 & 0 & \cdots\\
0 & A_2^* & B_3 & A_3 & \ddots\\
\vdots & \vdots & \ddots & \ddots & \ddots
\end{pmatrix}
\]
where each $A_m$ and $B_m$ is a $k\times k$ matrix for some fixed $k\in\bbN$ with $B_m=B_m^*$ and $\det(A_m)\neq0$ for all $m\in\bbN$.  An extensive introduction to the theory and applications of these operators can be found in \cite{DPS}, so we will only mention the facts that are directly relevant to our investigation.

In the context of our problem, a block Jacobi matrix is a bounded self-adjoint operator from $\ell^2(\bbN_0)$ to $\ell^2(\bbN_0)$ (where $\bbN_0=\bbN\cup\{0\}$) so its spectrum is a compact subset of the real line.  While the spectrum of such an operator is in general difficult to compute, one can easily verify that if $B_m\equiv0$ and $A_m\equiv I_{k\times k}$ for all $m\in\bbN$, then the spectrum of the corresponding block Jacobi matrix is $[-2,2]$.  In particular, the norm of this operator is $2$ in this case.

\section{Main Result}\label{Main}

With this basic knowledge of block Jacobi matrices in hand, we can now state our main result.
Suppose $\nu$ is as in Section \ref{wberg}, $s\in(0,\infty)$, and $n\in\bbN$ have been fixed.  Define the block Jacobi matrix $\mcj(\nu)$ by
\begin{align*}
&\mcj(\nu)_{n+k,k}=\mcj(\nu)_{k,n+k}=\begin{cases}
\frac{\gamma_{2k+2n+s}-\frac{\gamma_{2k+2n}\gamma_{2k+s}}{\gamma_{2k}}}{\sqrt{\gamma_{2k+2n}}\sqrt{\gamma_{2k+4n}-\frac{\gamma_{2k+2n}^2}{\gamma_{2k}}}}\qquad\qquad\qquad&  k=0,\ldots,n-1\\
\,\\
\frac{\left(\gamma_{2k+2n+s}-\frac{\gamma_{2k+2n}\gamma_{2k+s}}{\gamma_{2k}}\right)}{\sqrt{\gamma_{2k+2n}-\frac{\gamma_{2k}^2}{\gamma_{2k-2n}}}\sqrt{\gamma_{2k+4n}-\frac{\gamma_{2k+2n}^2}{\gamma_{2k}}}}& k\geq n
\end{cases}
\end{align*}
and all other entries of $\mcj(\nu)$ are equal to $0$.  
Note that the log-convexity of $\{\gamma_t\}_{t>0}$ (see \cite[Theorem 1.3.4]{NPBook}) implies that the quantities under the square roots in the entries of $\mcj(\nu)$ are all non-negative.  However, if we write
\begin{align*}
\gamma_{2k+4n}\gamma_{2k}-\gamma_{2k+2n}^2&=\int_{[0,1]^2}(x^{2k+4n}y^{2k}-x^{2k+2n}y^{2k+2n})d\mu(x)d\mu(y)\\
&=\int_{[0,1]^2}(xy)^{2k}x^{2n}(x^{2n}-y^{2n})d\mu(x)d\mu(y)
\end{align*}
and then symmetrize to obtain
\begin{align*}
\gamma_{2k+4n}\gamma_{2k}-\gamma_{2k+2n}^2&=\frac{1}{2}\int_{[0,1]^2}(xy)^{2k}(x^{2n}-y^{2n})^2d\mu(x)d\mu(y)>0,
\end{align*}
we see that the quantities under the square roots in the entries of $\mcj(\nu)$ are all strictly positive.  Similar reasoning shows the numerators of those expressions are also strictly positive.

Our main result can be stated as follows.

\begin{theorem}\label{better1}
Suppose $C\in\bbC$, $s\in(0,\infty)$, and $n\in\bbN$.  The operator $T_{z^n+C|z|^{s}}$ acting on $\mca^2_{\nu}(\bbD)$ is hyponormal if and only if $|C|\leq\|\mcj(\nu)\|^{-1}$.
\end{theorem}

The above discussion verifies that $\mcj(\nu)$ is a block Jacobi matrix.  Our next task is to prove the following result, which has clear implications for the application of Theorem \ref{better1}.

\begin{theorem}\label{bounded}
The spectrum of the operator $\mcj(\nu)$ consists of the interval $[-s/n,s/n]$ and an at most countable set of isolated points whose only accumulation points are among $\{\pm s/n\}$.  In particular, the operator $\mcj(\nu)$ is bounded.
\end{theorem}

The proof will require the following elementary lemma.

\begin{lemma}\label{sube}
It holds that
\[
\lim_{\kri}\left(\int_{[0,1]^2}(xy)^{k}(x^{2n}-y^{2n})^2d\mu(x)d\mu(y)\right)^{1/k}=1
\]
\end{lemma}

\begin{proof}
It is clear that
\[
\limsup_{\kri}\left(\int_{[0,1]^2}(xy)^{k}(x^{2n}-y^{2n})^2d\mu(x)d\mu(y)\right)^{1/k}\leq1
\]
Also notice that if $\delta\in(0,1)$ is fixed, then
\begin{align*}
&\liminf_{\kri}\left(\int_{[0,1]^2}(xy)^{k}(x^{2n}-y^{2n})^2d\mu(x)d\mu(y)\right)^{1/k}\\
&\qquad\qquad\qquad\geq\liminf_{\kri}\left(\int_{[1-\delta,1]^2}(xy)^{k}(x^{2n}-y^{2n})^2d\mu(x)d\mu(y)\right)^{1/k}\\
&\qquad\qquad\qquad\geq(1-\delta)^2\liminf_{\kri}\left(\int_{[1-\delta,1]^2}(x^{2n}-y^{2n})^2d\mu(x)d\mu(y)\right)^{1/k}\\
&\qquad\qquad\qquad=(1-\delta)^2
\end{align*}
Sending $\delta\rightarrow0$ proves the lemma.
\end{proof}

\begin{proof}[Proof of Theorem \ref{bounded}]
We will show that
\[
\lim_{\kri}\mcj(\nu)_{n+k,k}=\frac{s}{2n}
\]
This will show that $\mcj(\nu)$ is a compact perturbation of the block Jacobi matrix $\frac{s}{2n}(\mcl^n+\mcr^n)$, where $\mcl$ is the left shift and $\mcr$ is the right shift on $\ell^2(\bbN_0)$.  The operator $\frac{s}{2n}(\mcl^n+\mcr^n)$ has spectrum equal to $[-s/n,s/n]$, from which the desired conclusion follows.

Recall that for any $\eta>0$ it holds that
\[
\lim_{t\rightarrow\infty}\frac{\gamma_{t+\eta}}{\gamma_t}=1.
\]
This implies
\begin{equation}\label{jmodk}
\mcj(\nu)_{n+k,k}=(1+o(1))\frac{\gamma_{2k+2n+s}\gamma_{2k}-\gamma_{2k+2n}\gamma_{2k+s}}{\sqrt{\gamma_{2k+2n}\gamma_{2k-2n}-\gamma_{2k}^2}\sqrt{\gamma_{2k+4n}\gamma_{2k}-\gamma_{2k+2n}^2}}
\end{equation}
as $\kri$.  Now we write
\begin{align*}
\gamma_{2k+2n+s}\gamma_{2k}-\gamma_{2k+2n}\gamma_{2k+s}&=\int_{[0,1]^2}(x^{2k+2n+s}y^{2k}-x^{2k+2n}y^{2k+s})d\mu(x)d\mu(y)\\
&=\int_{[0,1]^2}(xy)^{2k}x^{2n}(x^s-y^s)d\mu(x)d\mu(y)
\end{align*}
Interchanging the roles of $x$ and $y$ and adding these expressions, we find
\begin{equation*}
\gamma_{2k+2n+s}\gamma_{2k}-\gamma_{2k+2n}\gamma_{2k+s}=\frac{1}{2}\int_{[0,1]^2}(xy)^{2k}(x^{2n}-y^{2n})(x^s-y^s)d\mu(x)d\mu(y)
\end{equation*}
Using similar reasoning on the expressions in the denominator of \eqref{jmodk}, we can rewrite the leading term of \eqref{jmodk} as
\begin{equation}\label{gform}
\frac{\int_{[0,1]^2}(xy)^{2k}(x^{2n}-y^{2n})(x^s-y^s)d\mu(x)d\mu(y)}{\sqrt{\left(\int_{[0,1]^2}(xy)^{2k-2n}(x^{2n}-y^{2n})^2d\mu(x)d\mu(y)\right)\left(\int_{[0,1]^2}(xy)^{2k}(x^{2n}-y^{2n})^2d\mu(x)d\mu(y)\right)}}
\end{equation}
We can bound \eqref{gform} from above by
\begin{equation}\label{upbound}
\frac{\int_{[0,1]^2}(xy)^{2k}(x^{2n}-y^{2n})(x^s-y^s)d\mu(x)d\mu(y)}{\int_{[0,1]^2}(xy)^{2k}(x^{2n}-y^{2n})^2d\mu(x)d\mu(y)}
\end{equation}
and from below by
\begin{equation}\label{lowbound}
\frac{\int_{[0,1]^2}(xy)^{2k}(x^{2n}-y^{2n})(x^s-y^s)d\mu(x)d\mu(y)}{\int_{[0,1]^2}(xy)^{2k-2n}(x^{2n}-y^{2n})^2d\mu(x)d\mu(y)}
\end{equation}
Lemma \ref{sube} implies that the denominators in \eqref{upbound} and \eqref{lowbound} decay subexponentially as $\kri$, and hence we obtain the same asymptotic behavior as $\kri$ if we replace each integral in the numerators of \eqref{upbound} and \eqref{lowbound} by the integral over $[1-\epsilon,1]^2$ for some $\epsilon>0$.

Now, fix some $\epsilon\in(0,1)$ and define
\[
g(\epsilon):=\max_{1-\epsilon\leq t\leq1}t^{s-2n}=\begin{cases}
(1-\epsilon)^{s-2n}\quad&\mbox{if}\quad s<2n\\
1 &\mbox{if}\quad s\geq2n.
\end{cases}
\]
If $Z>\frac{sg(\epsilon)}{2n}$, then $Zu^{2n}-u^s$ is an increasing function of $u$ on $[1-\epsilon,1]$.  Thus, when $Z>\frac{sg(\epsilon)}{2n}$ and $1\geq x\geq y\geq1-\epsilon$ it holds that $(x^s-y^s)<Z(x^{2n}-y^{2n})$.  It follows from \eqref{upbound} that for such a $Z$ we have
\begin{align*}
\limsup_{\kri}\mcj(\nu)_{n+k,k}\leq\limsup_{\kri}\frac{Z\int_{[1-\epsilon,1]^2}(xy)^{2k}(x^{2n}-y^{2n})^2d\mu(x)d\mu(y)}{\int_{[0,1]^2}(xy)^{2k}(x^{2n}-y^{2n})^2d\mu(x)d\mu(y)}=Z,
\end{align*}
where we used the fact that the denominator of the expression in \eqref{upbound} decays subexponentially as $\kri$.  Since $Z>\frac{sg(\epsilon)}{2n}$ was arbitrary, we conclude that
\[
\limsup_{\kri}\mcj(\nu)_{n+k,k}\leq\frac{sg(\epsilon)}{2n}.
\]
Taking $\epsilon\rightarrow0$, we obtain the desired upper bound.

By applying similar reasoning, we see that if $h(\epsilon):=\min_{1-\epsilon\leq t\leq1}t^{s-2n}$ and $Z'<\frac{sh(\epsilon)}{2n}$, then
\begin{align*}
\liminf_{\kri}\mcj(\nu)_{n+k,k}&\geq\liminf_{\kri}\frac{Z'\int_{[1-\epsilon,1]^2}(xy)^{2k}(x^{2n}-y^{2n})^2d\mu(x)d\mu(y)}{\int_{[0,1]^2}(xy)^{2k-2n}(x^{2n}-y^{2n})^2d\mu(x)d\mu(y)}\\
&\geq\liminf_{\kri}\frac{Z'(1-\epsilon)^{4n}\int_{[1-\epsilon,1]^2}(xy)^{2k-2n}(x^{2n}-y^{2n})^2d\mu(x)d\mu(y)}{\int_{[0,1]^2}(xy)^{2k-2n}(x^{2n}-y^{2n})^2d\mu(x)d\mu(y)}\\
&=Z'(1-\epsilon)^{4n}.
\end{align*}
Since $Z'<\frac{sh(\epsilon)}{2n}$ was arbitrary, we conclude that
\[
\liminf_{\kri}\mcj(\nu)_{n+k,k}\geq\frac{sh(\epsilon)(1-\epsilon)^{4n}}{2n}
\]
Taking $\epsilon\rightarrow0$, we obtain the desired lower bound.
\end{proof}

As an example of Proposition \ref{bounded}, consider the case when $\mu=2rdr$.  In this case the measure $\nu$ is normalized area measure on $\bbD$ and $\gamma_{t}=2(t+2)^{-1}$ so $\mcj(dA)$ is
\begin{align*}
&\mcj(dA)_{n+k,k}=\mcj(dA)_{k,n+k}=\begin{cases}
\frac{s\sqrt{(k+n+1)(k+2n+1)}}{2(k+1+s/2)(k+n+1+s/2)}\qquad\qquad\qquad&  k=0,\ldots,n-1\\
\,\\
\frac{s(k+1)\sqrt{(k+n+1)(k+2n+1)}}{2n(k+1+s/2)(k+n+1+s/2)}& k\geq n.
\end{cases}
\end{align*}
We see that the conclusion of Theorem \ref{bounded} holds true for this matrix.  Furthermore, one can quickly verify by hand that if $s\geq2n$, then each non-zero entry of $\mcj(dA)$ is less than $\frac{s}{2n}$.  This implies that when $s\geq2n$ the spectrum of $\mcj(dA)$ is precisely $[-s/n,s/n]$ and so Theorem \ref{better1} implies \cite[Theorem 2]{SimHypo}.  



\section{Proof of Theorem \ref{better1}}

Throughout this section, let us suppose that $n\in\bbN$, $s\in(0,\infty)$, and $\nu$ as in Section \ref{wberg} are fixed.  As in \cite{FL,SimHypo}, we will use the formula
\begin{align}\label{expand}
 &\left\langle \left[(T+S)^{*},T+S\right]u,u\right\rangle \nonumber \\
 &\quad =\left\langle Tu,Tu\right\rangle -\left\langle T^{*}u,T^{*}u\right\rangle +2\mathrm{Re}\left[\left\langle Tu,Su\right\rangle -\left\langle T^{*}u,S^{*}u\right\rangle \right]+\left\langle Su,Su\right\rangle -\left\langle S^{*}u,S^{*}u\right\rangle
\end{align}
with $T=T_{z^n}$ and $S=T_{C|z|^s}$.
The first step in our proof will be the following adaptation of \cite[Lemma 1]{SimHypo} to weighted Bergman spaces.

\begin{lemma}\label{bproj}
If $k\in\bbN_0:=\bbN\cup\{0\}$ and $t\in(0,\infty)$, then
\[
P_{\nu}(z^k|z|^t)=\frac{\gamma_{2k+t}}{\gamma_{2k}}z^k,\qquad\quad\mbox{and}\qquad\quad P_{\nu}(\bar{z}^{k+1}|z|^t)=0.
\]
\end{lemma}

\begin{proof}
A calculation shows that
\[
\left\langle z^q, \frac{\gamma_{2k+t}}{\gamma_{2k}}z^k\right\rangle=\left\langle z^q, z^k|z|^t\right\rangle
\]
for every $q\in\bbN_0$, so the first claim follows from the fact that polynomials are dense in $\mca_\nu^2(\bbD)$.  The second claim follows from a similar calculation.
\end{proof}

With Lemma \ref{bproj} in hand, we consider $u=\sum_{k=0}^{\infty}u_kz^k\in\mca^2_{\nu}(\bbD)$ and calculate
\begin{align*}
\langle T_{z^n}u,T_{z^n}u\rangle&=\displaystyle\sum_{k=0}^{\infty}|u_k|^2\gamma_{2k+2n}\\
\langle T_{\bar{z}^n}u,T_{\bar{z}^n}u\rangle&=\displaystyle\sum_{k=n}^{\infty}\frac{\gamma_{2k}^2}{\gamma_{2k-2n}}|u_k|^2\\
\mathrm{Re}[\langle T_{z^n}u,T_{C|z|^{s}}u\rangle-\langle T_{\bar{z}^n}u,T_{\bar{C}|z|^{s}}u\rangle]&=\displaystyle\sum_{k=0}^{\infty}\mathrm{Re}[u_k\bar{u}_{k+n}\bar{C}]\left(\gamma_{2k+2n+s}-\frac{\gamma_{2k+2n}\gamma_{2k+s}}{\gamma_{2k}}\right),
\end{align*}
Notice that $\gamma_{2k+2n+s}-\frac{\gamma_{2k+2n}\gamma_{2k+s}}{\gamma_{2k}}$ is the numerator of the entry $\mcj(\nu)_{n+k,k}$ and hence the discussion before Theorem \ref{better1} implies each of these terms is strictly positive.  That discussion also implies $\gamma_{2k+2n}>\frac{\gamma_{2k}^2}{\gamma_{2k-2n}}$ when $k\geq n$.  Thus we may reason as in \cite{SimHypo} and conclude that $T_{z^n+C|z|^s}$ is hyponormal if and only if
\begin{align}\label{cinf}
|C|\leq\inf\left\{\frac{\displaystyle\sum_{k=0}^{n-1}u_k^2\gamma_{2k+2n}+\displaystyle\sum_{k=n}^{\infty}u_k^2\left(\gamma_{2k+2n}-\frac{\gamma_{2k}^2}{\gamma_{2k-2n}}\right)}{2\displaystyle\sum_{k=0}^{\infty}u_ku_{k+n}\left(\gamma_{2k+2n+s}-\frac{\gamma_{2k+2n}\gamma_{2k+s}}{\gamma_{2k}}\right)}\right\},
\end{align}
where the infimum is taken over all non-negative sequences $\{u_k\}$ satisfying $\sum_{k=0}^{\infty}|u_k|^2\gamma_{2k}<\infty$ that are not the zero sequence.  For convenience, we will restate the condition \eqref{cinf} as
\begin{equation}\label{in}
\kappa:=\sup\left\{\frac{2\displaystyle\sum_{k=0}^{\infty}u_ku_{k+n}\left(\gamma_{2k+2n+s}-\frac{\gamma_{2k+2n}\gamma_{2k+s}}{\gamma_{2k}}\right)}{\displaystyle\sum_{k=0}^{n-1}u_k^2\gamma_{2k+2n}+\displaystyle\sum_{k=n}^{\infty}u_k^2\left(\gamma_{2k+2n}-\frac{\gamma_{2k}^2}{\gamma_{2k-2n}}\right)}\right\}\leq\frac{1}{|C|}
\end{equation}
Now we make the substitution
\[
v_j=\begin{cases}
u_j\sqrt{\gamma_{2k+2n}}\qquad\qquad\qquad & j\leq n-1\\
u_j\sqrt{\gamma_{2k+2n}-\frac{\gamma_{2k}^2}{\gamma_{2k-2n}}} & j\geq n
\end{cases}
\]
in \eqref{in} (we used the discussion before Theorem \ref{better1} here).  This gives
\begin{equation}\label{in2}
\kappa=\sup\left\{\frac{2\left(\displaystyle\sum_{k=0}^{n-1}\frac{v_kv_{k+n}\left(\gamma_{2k+2n+s}-\frac{\gamma_{2k+2n}\gamma_{2k+s}}{\gamma_{2k}}\right)}{\sqrt{\gamma_{2k+2n}}\sqrt{\gamma_{2k+4n}-\frac{\gamma_{2k+2n}^2}{\gamma_{2k}}}}+\displaystyle\sum_{k=n}^{\infty}\frac{v_kv_{k+n}\left(\gamma_{2k+2n+s}-\frac{\gamma_{2k+2n}\gamma_{2k+s}}{\gamma_{2k}}\right)}{\sqrt{\gamma_{2k+2n}-\frac{\gamma_{2k}^2}{\gamma_{2k-2n}}}\sqrt{\gamma_{2k+4n}-\frac{\gamma_{2k+2n}^2}{\gamma_{2k}}}}\right)}{\displaystyle\sum_{k=0}^{\infty}v_k^2}\right\}
\end{equation}
and we take the supremum over all non-negative $\{v_k\}_{k=0}^{\infty}$ that satisfy
\[
0<\sum_{k=0}^{\infty}v_k^2\frac{\gamma_{2k}}{\gamma_{2k+2n}-\frac{\gamma_{2k}^2}{\gamma_{2k-2n}}}<\infty.
\]
Notice that our assumptions on $\nu$ imply
\[
\lim_{k\rightarrow\infty}\left[\frac{\gamma_{2k}}{\gamma_{2k+2n}-\frac{\gamma_{2k}^2}{\gamma_{2k-2n}}}\right]=\infty,
\]
so in particular
\[
\ell^2_{\gamma}(\bbN_0):=\left\{\{v_k\}_{k=0}^{\infty}:\sum_{k=0}^{\infty}|v_k|^2\frac{\gamma_{2k}}{\gamma_{2k+2n}-\frac{\gamma_{2k}^2}{\gamma_{2k-2n}}}<\infty\right\}\subseteq\ell^2(\bbN_0)
\]
and in fact $\ell^2_{\gamma}(\bbN_0)$ is dense in $\ell^2(\bbN_0)$ in the $\ell^2(\bbN_0)$-metric because $\ell^2_{\gamma}(\bbN_0)$ contains all finite sequences.  Therefore, by the scale invariance of the expression in \eqref{in2}, we may define $\kappa$ by taking the supremum only over those non-negative sequences $\{v_k\}_{k=0}^{\infty}\in\ell^2_{\gamma}(\bbN_0)$ such that $\sum v_k^2=1$.

Since $\|\mcj(\nu)\|<\infty$ (by Theorem \ref{bounded}) and hence self-adjoint, we can use the well-known formula
\begin{equation}\label{jnorm}
\|\mcj(\nu)\|=\sup_{\|x\|=1}\langle x,\mcj(\nu)x\rangle_{\ell^2(\bbN_0)}
\end{equation}
(see \cite[page 216]{RS1}), where the supremum is taken over all $x\in\ell^2(\bbN_0)$ with norm $1$ in this space.  By the density property we just mentioned, it suffices to take the supremum over all $x\in\ell^2_{\gamma}(\bbN_0)$ with norm $1$ in $\ell^2(\bbN_0)$.  Since all the entries of $\mcj(\nu)$ are positive real numbers, we recognize that the right-hand side of \eqref{jnorm} is equal to the right-hand side of \eqref{in2}.  We conclude that $\|\mcj(\nu)\|=\kappa$ as desired.


\end{document}